\documentclass{svmult}
\usepackage[utf8]{inputenc}
\usepackage[T1]{fontenc}
\usepackage{amsmath}
\usepackage{amsfonts}
\usepackage{color}
\usepackage{graphicx}
\usepackage{makeidx}
\usepackage{multicol}
\usepackage[bottom]{footmisc}

\newcommand{\bitem}{\begin{itemize}}
\newcommand{\eitem}{\end{itemize}}

\newcommand{\bapp}{\begin{application}}
\newcommand{\eapp}{\end{application}}

\newcommand{\bapps}{\begin{applications}}
\newcommand{\eapps}{\end{applications}}

\newcommand{\bdefi}{\begin{definition}}
\newcommand{\edefi}{\end{definition}}

\newcommand{\beq}{\begin{equation}}
\newcommand{\eeq}{\end{equation}}

\def\bpm{\begin{pmatrix}}
\def\epm{\end{pmatrix}}

\newcommand{\bcas}{\begin{cases}}
\newcommand{\ecas}{\end{cases}}

\newcommand{\bex}{\begin{exemp}}
\newcommand{\eex}{\end{exemp}}

\newcommand{\bexs}{\begin{exemps}}
\newcommand{\eexs}{\end{exemps}}

\newcommand{\beqna}{\begin{eqnarray}}
\newcommand{\eeqna}{\end{eqnarray}}

\newcommand{\beqnas}{\begin{eqnarray*}}
\newcommand{\eeqnas}{\end{eqnarray*}}

\definecolor{green}{rgb}{0,.7,.2}
\definecolor{orange}{rgb}{0.9,.5,0}

\newcommand{\LL}{{\rm L}}

\def\tr{\textmd{trace}\,}

\def\det{{ \rm{det}}}  

\def\Id{{\rm{Id}}} 

\def\cA{{\mathcal A }}

\def\cM{{\mathcal M }}
\def\cP{{\mathcal P }}

\def\bbC{{\mathbb{C}}}

\def\bbH{{\mathbb{H}}}

\def\bbO{{\mathbb{O}}}
\newcommand{\bbR}{{\mathbb {R}}}
\newcommand{\bbS}{{\mathbb {S}}} 

\makeindex

\begin{document}

\title*{Dyson processes on the octonion algebra}
\author{Songzi Li\inst{1}}
\institute{Institut de Math\'ematiques de Toulouse, Universit\'e Paul Sabatier,
118 route de Narbonne, 31062 Toulouse, France;
School of Mathematical Sciences, Fudan University, 220 Handan Road, 200433 Shanghai, China;
\texttt{songzi.li@math.univ-toulouse.fr}
}

\maketitle

%
%
\abstract{We consider Brownian motion on  symmetric matrices of octonions, and study the law of the spectrum. Due to the fact that the octonion algebra is  nonassociative, the dimension of the matrices plays a special role. We provide two specific models on octonions, which give some indication of the relation between the multiplicity of eigenvalues and the exponent in the law of the spectrum.}

\section{Introduction}

The study of the laws of the spectrum is one of the most important topics in random matrix theory. One may consider stochastic diffusion  processes on specific set of matrices, for example symmetric or Hermitian matrices.   Usually one  considers the empirical measure of the spectrum, which is often again a stochastic diffusion process,  called a Dyson process,  see the works of Wigner \cite{Wig58}, Mehta \cite{Meh}, Dyson \cite{dyson}, Anderson-Guionnet-Zeitouni \cite{AndGuioZeit}, Erd\"os and co \cite{erdosandco3, erdosandco2, erdosandco1} , Forrester \cite{Forrester} and references therein.

Let us recall some classical results on this topic. Consider specific matrices with  independent Gaussian elements  : real symmetric ($\beta = 1$), Hermitian ($\beta = 2$) and real quaternionic ($\beta = 4$). Then the law of their eigenvalues $(\lambda_{i})_{1\leq i \leq n}$, ordered as $\lambda_{1} \leq ... \leq \lambda_{n}$, has a density with respect to the Lebesgue measure $d\lambda_{1}...d\lambda_{n}$ which is  \beq\label{factor} C_{\beta, n}{\rm exp}(-\frac{\beta}{2}\sum^{n}_{j=1}\lambda^{2}_{j})\prod_{1 \leq j \leq k \leq n}|\lambda_{k} - \lambda_{j}|^{\beta},\eeq where $C_{\beta, n}$ are constants depending on $\beta$, $n$.

On the other hand, one can also consider random matrices with stochastic process as entries. In his paper~\cite{dyson}, Dyson derived the stochastic equations of the eigenvalues of Hermitian matrices whose elements are independent complex Brownian motions (see also in this direction   Anderson-Guionnet-Zeitouni \cite{AndGuioZeit}, Mehta \cite{Meh}, Li-Li-Xie \cite{LLX}). The stochastic process on the spectrum provides a dynamic way to study the law of the eigenvalues of the matrix with Gaussian entries. In fact there are two ways: one is through the law of the eigenvalues of matrices with Brownian motions as entries, considered at  time $t = 1$; the other one is through the matrix whose elements are Ornstein-Uhlenbeck process, since when $t \rightarrow \infty$ the law of the matrix converges to a matrix with  Gaussian entries, and the law of its spectrum is invariant through O-U process: in this case, the law of the spectrum may be seen as the invariant (in fact reversible, see definition \ref{symmopt}) law of the process. This is in general a much easier way to identify the law, since reversible measures are easy to identify through the knowledge of the generator.

Meanwhile, if we consider real symmetric matrix, Hermitian matrix and real quaternionics matrix  as real ones of dimension respectively $n \times n$, $2n \times 2n$, $4n \times 4n$, the multiplicity of the eigenvalues is again 1, 2 and 4. This fact leads us to wonder whether this exponent  factor in the density reflects the multiplicity of the eigenvalues. However, this is  not true.

In a recent paper, Bakry and Zani \cite{BakZ}, the authors considered real symmetric matrices whose  elements are independent Brownian motions depending on  some associative algebra structure of the Clifford type. Their computation of the law of the spectrum shows that, even though there is still the term $\prod_{1 \leq j \leq k \leq n}|\lambda_{k} - \lambda_{j}|^{\beta }$ with $\beta = 1, 2, 4$, the factor $\beta$ here reflects the structure of the algebra, known as Bott periodicity,  rather than the dimension of the eigenspaces, which in this situation may be as large as we want.

The previous study on Dyson Brownian motion, including the work of Bakry and Zani \cite{BakZ} on Clifford algebra, mainly concentrated on the case where  the underlying algebra is associative. It is therefore worth understanding how important this property is in the study of the related Dyson processes. The octonion algebra, which is nonassociative but only alternative, provides a good example for us to start with. Its structure differs from the Clifford one, although Clifford algebras with 1 or 2 generators coincide with complex numbers  and quaternions, the Clifford algebra with 3 generators does not coincide with octonions, even if the algebras have the same real dimension 8. In his book \cite{Forrester}, section $1.3.5$, Forrester mentions that the distribution~\eqref{factor} with $\beta = 8$ can be realized by $2 \times 2$ matrices on octonions, with Gaussian entries.  It is therefore worth to look at the associated Dyson process, which could also provide this result through the study of its reversible measure.

There are only four normed division algebras: $\bbR$, $\bbC$, $\bbH$ and $\bbO$. We are familiar with $\bbR$, $\bbC$, and while  the quaternion algebra  $\bbH$ is noncommutative but associative,the octonion algebra $\bbO$  is  nonassociative,  but only alternative. Even though their properties are not so nice, octonions have some important connections to different fields of mathematics, such as geometry, topology and algebra. One interesting example is its role in the classification of simple Lie algebra. There are 3 infinite families of simple Lie algebras, coming from the isometry groups of the projective spaces $\mathbb{RP}^{n}$, $\mathbb{CP}^{n}$ and $\mathbb{HP}^{n}$. The remaining 5 simple Lie algebras were later discovered to be in connection with octonions: they come from the isometry groups of the projective planes over $\bbO$, $\bbO \otimes \bbC$, $\bbO \otimes \bbH$, $\bbO \otimes \bbO$ and the automorphism group of octonions. It is also worth to mention that, according to the independent work by Kervaire~\cite{Kerv} and Bott-Milnor~\cite{BottMilnor} in 1958, there are only 4 parallelizable spheres: $\bbS^{0}$, $\bbS^{1}$, $\bbS^{3}$ and $\bbS^{7}$, which correspond precisely to elements of unit norm in the normed division algebras of the real numbers, complex numbers, quaternions, and octonions. See more examples in the paper by J. Baez \cite{BaezJ}.

For the eigenvalue problem of matrices on octonions, Y.G.Tian proved in his paper~\cite{tian} that $2 \times 2$ Hermitian matrix on octonions has 2 eigenvalues, each of them has multiplicity 8. For $3 \times 3$ Hermitian octonionic matrix, Dray-Manogue~\cite{drayM} and Okubo~\cite{oku} showed that it has 6 eigenvalues with multiplicity 4. For $4 \times 4$ and $5 \times 5$  Hermitian octonionic matrices, there are only numerical results, indicating that the eigenvalues have multiplicity 2~\cite{tian}. It is still unknown for matrices in higher dimension. Following the analysis of Bakry and Zani ~\cite{BakZ}, one may expect that  the study of probabilistic models on matrices of octonions could  give new insights in these directions.

In this paper, we consider Brownian motions on symmetric matrices of octonions. Due to the fact that octonions are nonassociative, and in contrast with the Clifford case, the dimension of the matrices plays a specific role. In fact, contrary to the real, complex and quaternionic cases, octonions do not give rise to infinite series of Lie groups but only specific ones, which are closely related to dimension 2. Thus the study of Dyson processes is mainly pertinent in this dimension, although we introduce another probabilistic model related to the octonion algebra, but with a special structure, see Section 4.2. To study the law of the spectrum of the matrices, we consider the processes on the characteristic polynomials $P(X)$, as introduced in the paper by Bakry and Zani \cite{BakZ}. Because of the specific structure of octonions, the traditional way to compute the law of the spectrum turns out to be quite hard,  while computation on the process of $P(X)$ provides a simpler and more efficient method to see things clearly.

The paper is organized as follows. Section 2 gives an introduction to the basics of the octonion algebra; Section 3 explains briefly the language and tools of symmetric diffusion process; Section 4 states  our main results, two specific models on octonions, and then we explain what is so special about dimension 2; Section 5 is devoted to the demonstration of the connection between the algebra structure and the Euclidean structure associated with the associated symmetric matrices, and the fact that the the  two exponents,  multiplicity of eigenvalues and  exponent in the law of eigenvalues,  are not correlated.

\section{The octonion algebra}

In this section, we recall some facts about the octonion algebra,  and we refer to~\cite{BaezJ} for more details. We start with a few definitions.
\bdefi
An algebra A is a division algebra if for any $a, b \in A$, with $ab = 0$, then either $a = 0$ or $b = 0$. A normed division algebra is a division algebra that is also a normed vector space with $\| ab\|= \|a\| \|b \|$.
\edefi

\bdefi
An algebra $A$ is alternative if the subalgebra generated by any two elements is associative. By a theorem of  Artin~\cite{scha}, this is equivalent to the fact that for any $a, b \in A$, $(aa)b = a(ab), \ \ \ (ba)a = b(aa)$.
\edefi

As mentioned earlier, there are only four normed division algebras, $\bbR$, $\bbC$, $\bbH$, $\bbO$. There is a nice way called "Cayley-Dickson construction" to produce this sequence of algebras: the complex number $a + ib$ can be seen as a pair of real numbers $(a,b)$; the quaternions can be defined as a pair of complex number; and similarly the octonions is a pair of quaternions. As the construction proceeds, the property of the algebra becomes worse and worse: the quaternions are noncommutative but associative,  while the octonions are only alternative but not associative.

Since octonions and Clifford algebra are both the algebra with  dimension $2^{n}$ (in this case $n=3$),  which  share some special property, we can use the presentation provided in Bakry and Zani \cite{BakZ}  to describe the algebra structure on a basis of octonions, in order to simplify the computations. This presentation is not classical, and we shall therefore use the table below.

Define $E = \{1, 2, 3\}$, and let $ \mathcal{P}(E)$ denote the set of   the subsets of $E$. For every set $A \in \mathcal{P}(E)$,we associate a  basis element  $\omega_{A}$ in the octonion algebra, with  $\omega_{\emptyset} = \Id$, the identity element. Then an element $x \in \bbO$ can be written in the form
$$
x = \sum_{A}x_{A}\omega_{A}, ~x_A\in \bbR,
$$
and the product of two elements $x$ and $y$ is given by
$$
x y = \sum_{A, B}x_{A}y_{B}\omega_{A}\omega_{B}.
$$
It remains to define  $\omega_{A}\omega_{B}$ for given $A, B \in  \mathcal{P}(E)$, which is given through the following rule: denote by $A.B$ the symmetric difference $A \cup B\setminus(A \cap B)$, then $\omega_{A}\omega_{B} = (A|B)\omega_{A.B}$, where $(A|B)$ takes value in $\{ -1, 1\}$.  Then, the multiplication rule in the octonion algebra is defined by a sign table,  which is as follows :
\\
$$
\begin{tabular}{|c|c|c|c|c|c|c|c|c|}
  \hline
                & $\emptyset$ & $\{1\}$ & $\{2\}$ & $\{3\}$ & $\{1,2\}$ & $\{1, 3\}$ & $\{2, 3\}$ & $\{1, 2, 3\}$ \\
  \hline
  $\emptyset$  & 1 & 1 & 1 & 1 & 1 & 1 & 1 & 1 \\\hline
  $\{1\}$      & 1 & -1 & 1 & 1 & -1 & -1 & 1 & -1 \\\hline
  $\{2\}$      & 1 & -1 & -1 & 1 & 1 & -1 & -1 & 1\\\hline
  $\{3\}$      & 1 & -1 & -1 & -1 & 1 & 1 & 1 & -1\\\hline
  $\{1,2\}$    & 1 & 1 & -1 & -1 & -1 & -1 & 1 & 1\\\hline
  $\{1, 3\}$   & 1 & 1 & 1 & -1 & 1 & -1 & -1 & -1\\\hline
  $\{2, 3\}$   & 1 & -1 & 1 & -1 & -1 & 1 & -1 & 1\\\hline
  $\{1, 2, 3\}$& 1 & 1 & -1 & 1 & -1 & 1 & -1 & -1\\
  \hline
\end{tabular}
$$
\\
In this table, the element $(i, j)$ is the sign $(A_i|A_j)$, where $A_i$ is the $i$th element in the first column, $A_j$ is the $j$th element in the first row.

From the facts that for  $A, B \neq \emptyset$, $\omega^{2}_{A} = -1$ and $\omega_{A}\omega_{B} = -\omega_{B}\omega_{A}$, it is easy to get the following rules:
\begin{eqnarray*}
(A|A) &=& \left\{
            \begin{array}{ll}
              -1, & \hbox{$A \neq \emptyset$,}   \\
              1, & \hbox{$A = \emptyset$, }
            \end{array}
          \right.\\
(A|B) &=& -(B|A), \ \ \  for \ \ \ B \neq A, A,B\neq \emptyset.\\
\end{eqnarray*}

It can be seen from the above table that $\bbO$ is an algebra, non-associative but alternative. Moreover  $\bbO$ can be equipped with the Euclidean structure obtained by identifying $\bbO$ as a 8 dimensional (real) vector space via
\[x =\sum_A x_A \omega_A \mapsto (x_\emptyset, x_{\{1\}}, x_{\{2\}}, x_{\{3\}}, x_{\{1,2\}}, x_{\{1,3\}}, x_{\{2,3\}}, x_{\{1,2,3\}})\,,\]
so that the inner product  and the norm are respectively :
\[\langle x, y\rangle = \sum_A x_A y_A \  \ , \ \   \Vert x\Vert = (\sum_A x_A^2)^{1/2}\,,\]
so that $\{\omega_A,~ A\in \cP(E)\}$ form a real  orthonormal basis for the algebra $\bbO$.

Let us recall that to prove that  $\bbO$ is a division algebra, it is usual to  introduce the conjugate
\[x = \sum_A x_A \omega_A \mapsto x^* = \sum_A x_A \omega_A (A |A),\]
and observe that $(xy)^* = y^* x^*$ , $xx^* = x^*x$ and $\Vert x\Vert^2 = xx^*$,
so that $\Vert xy\Vert^2 = (xy)(xy)^*= (xy)(y^*x^*) = x(yy^*)x^* = \Vert x\Vert^2 \Vert y\Vert^2$.

Altough the previous table does not provide an associative algebra, the octonion algebra satisfies however some useful identities. In what follows, we shall make a strong use of  Moufang identities, which are stated as follows :  for elements $x, y, z$ belongs to $\bbO$, we have
\beqnas
z(x(zy)) &=& (zxz)y,\\
((xz)y)z &=& x(zyz),\\
(zx)(yz) &=& (z(xy))z,\\
(zx)(yz) &=& z((xy)z).\\
\eeqnas

We shall mainly use this for the elements $\omega_A\in \bbO$, although Moufang identities provide more information than this.

According to the alternativity property, we can get some basic formulae about the sign table $\{(A|B)\}$ for octonions,

\begin{lemma}\label{altern}
For $A, B, C, D \in \mathcal{P}(E)$, we have
\begin{enumerate}
\item{ \label{oct.it1}
$(A.B |B) = (A|B)(B|B)$.}
\item{ \label{oct.it2}
$(A .B |A)(A.B |B) = (A.B|A.B)$.}
\item{\label{oct.it3} If $A.B \neq \emptyset$,
$(A.C |A)(B.C |B) = -(A.C |B)(B.C |A)$.}
\item{\label{oct.it4}If $A.B.C.D = \emptyset$,
$$
(B.C |C)(C.D |D)(D.A |A)(A.B |B) = (B.D |B.D).
$$}
\end{enumerate}
\end{lemma}

\begin{proof}
The first one is just the result of alternativity:\\
$$
(\omega_{A}\omega_{B})\omega_{B} = (A|B)\omega_{A.B}\omega_{B} = (A.B |B)\omega_{A},
$$
while
$$
(\omega_{A}\omega_{B})\omega_{B} = \omega_{A}(\omega_{B}\omega_{B}) = (B|B)\omega_{A},
$$
Hence $$(A.B |B) = (A|B)(B|B).$$
The second one can be easily proved by the first statement.

For the third statement, we first remark that,  for any vector $x = \sum x_{C}\omega_{C}$, $x\omega_{A}$ is always orthogonal to $x\omega_{B}$ if $A \neq B$. Indeed, to see this,  we may reduce to the case where $\|x\|=1$, and then observe that the fact that the algebra is a division algebra shows that for any $y \in \bbO$, $y\mapsto xy$ is an orthogonal transformation. For $A, B \in \cP(E)$, $A.B \neq \emptyset$, and any $C \in \cP(E)$, choose $D = A.B.C$ and set $x = \omega_C + \omega_D$. Then
\beqnas
0 &=& \langle x\omega_{A}, x\omega_{B} \rangle\\
&=& \langle (C|A)\omega_{C.A} + (D|A)\omega_{D.A}, (C|B)\omega_{C.B} + (D|B)\omega_{D.B} \rangle\\
&=&  \langle (C|A)\omega_{C.A}, (D|B)\omega_{D.B} \rangle + \langle (D|A)\omega_{D.A}, (C|B)\omega_{C.B} \rangle\\
&=&  (C|A)(D|B) + (D|A)(C|B).\\
\eeqnas
Since $D = A.B.C$, the above formula indicates that for $A, B, C \in \cP(E)$, $A.B \neq \emptyset$,
$$
(C|A)(A.B.C|B) + (A.B.C|A)(C|B) =0 .
$$
By changing $C$ into $A.C$, we get
$$
(A.C|A)(B.C|B) + (B.C|A)(A.C|B)=0,
$$
then the third statement is proved.

For the last statement, denote
\beq\label{theta}
\Theta:=(B.C |C)(C.D |D)(D.A |A)(A.B |B).
\eeq Then,  from Moufang identities,
\begin{eqnarray*}
((\omega_{B.C}\omega_C)(\omega_{D.A}\omega_A))((\omega_{C.D}\omega_D)(\omega_{A.B}\omega_B)) &=& (B.C |C)(C.D |D)(D.A |A)(A.B |B)(\omega_B\omega_D)(\omega_C\omega_A)\\
&=& \Theta(B|D)(C|A)\omega_{B.D}\omega_{C.A}\\
&=& \Theta(B|D)(C|A)\omega_{B.D}\omega_{B.D}\\
&=& \Theta(B|D)(C|A)(B.D|B.D).\\
\end{eqnarray*}
On the other hand,
\begin{eqnarray*}
\omega_A\omega_B = (A.B|A.B)(A|A)(B|B)\omega_B\omega_A.\\
\end{eqnarray*}
\begin{eqnarray*}
& &((\omega_{B.C}\omega_C)(\omega_{D.A}\omega_A))((\omega_{C.D}\omega_D)(\omega_{A.B}\omega_B))\\
&=& (D|D)(B.C|B.C)(A|A)(A|A)(B|B)(A.B|A.B)((\omega_{B.C}\omega_C)(\omega_A\omega_{B.C}))((\omega_{A.B}\omega_D)(\omega_B\omega_{A.B}))\\
&=& (D|D)(B.C|B.C)(B|B)(A.B|A.B)(\omega_{B.C}(\omega_C\omega_A)\omega_{B.C})(\omega_{A.B}(\omega_D\omega_B)\omega_{A.B})\\
&=& (D|D)(B.C|B.C)(B|B)(A.B|A.B)(C|A)(D|B)(\omega_{B.C}\omega_{C.A}\omega_{B.C})(\omega_{A.B}\omega_{D.B}\omega_{A.B})\\
&=& (D|D)(B|B)(C|A)(D|B)\omega_{C.A}\omega_{D.B}\\
&=& (D|D)(B|B)(C|A)(D|B)(B.D|B.D).\\
\end{eqnarray*}
Hence,
\begin{eqnarray*}
\Theta &=& (B|D)(C|A)(B.D|B.D)(D|D)(B|B)(C|A)(D|B)(B.D|B.D)\\
&=& (B|D)(D|D)(B|B)(D|B)\\
&=& (B.D|B.D),\\
\end{eqnarray*}
which ends the proof of the lemma.
\end{proof}

For a $n \times n$ matrix on octonions, write it as $\mathcal{M} = \sum_{A}M^{A}\omega_{A}$, where $\{M^{A}\}$ are real  $n \times n$ matrices. For  an $n$ dimensional  vector $\sum_{B}X^{B}\omega_{B}$,
\begin{eqnarray*}
(\sum_{A}M^{A}\omega_{A})(\sum X^{B}\omega_{B}) = \sum_{A, B}M^{A}X^{B}(A|B)\omega_{A.B} = \sum_{A, B}(A.B|B)M^{A.B}X^{B}\omega_{A}.
\end{eqnarray*}
Therefore, $\mathcal{M}$ can be expressed by the real  $8n\times 8n$ block matrix $\{M^{A, B}_{ij}\}$, where $M^{A, B}_{ij} = (A.B|B)M_{ij}^{A.B}$.

This leads to the following definition:

\bdefi\label{defi.matrix.octonions} A $(2^3\times n)\times (2^3\times n)$ block matrix $M^{A,B}$ (where $A, B\subset\{1,2,3\}$) is a real octonionic if  $M^{A,B}= (A.B|B) M^{A.B}$, where $M^{A}= M^{A, \emptyset}$ is a family of 8  $n\times n$ square matrices. It is the real form of a matrix with octonionic entries. We shall denote it as $\cM= \sum_A  M^A\omega_A$.

\edefi

Then, we shall say that an  octonionic matrix is symmetric if its real form is symmetric. This corresponds to the fact that, for any $A\in \cP(E)$, $(M^{A})^{t} = (A|A)M^{A}$.

That is to say, $(M^{A, B})^{t} =(A.B|B) (M^{A.B})^{t} = M^{B, A} = (B.A|A)M^{A.B}$. Due to property 2 of Lemma 2.3, this leads to the fact  that for any $A\in \cP(E)$, $(M^{A})^{t} = (A|A)M^{A}$, i.e. $M^\emptyset$ is symmetric while $M^A$ is antisymmetric for any $A\neq \emptyset$.

It is worth to point out that since the octonion algebra is   not associative, there is no matrix presentation of the algebra structure for  the octonions, and therefore the matrix multiplication of the real octonionic matrices does not corresponds to the octonionic  multiplication of the associated matrices with octonion entries. Even the product of octonionic matrices is not octonionic in general.

The inverse of an octonionic matrix is in general not octonionic, and  its exact structure is not easy to decipher; the octonionic property may not be preserved. The following lemma gives a condition for this last property to hold, and will play an important role in the rest of this paper.

\begin{lemma}\label{lemma.invert}
Let $M = \sum M^{A}\omega_{A}$ be an octonionic matrix  such that $M^\emptyset $ is invertible. Assume moreover that,  for any $A, B \in \mathcal{P}(E)$
\beq\label{symm}
M^{A}(M^{\emptyset})^{-1}M^{B} = M^{B}(M^{\emptyset})^{-1}M^{A},
\eeq and that
$\sum_{C}M^{C}(M^{\emptyset})^{-1}M^{C}$ is invertible.
Then, $M$ is invertible and its inverse $N$ is octonionic,
satisfying $N= \sum_A \omega_A N^A$, with
\beq\label{eqN1}
N^{A} = -N^{\emptyset}M^{A}(M^{\emptyset})^{-1}, \ \ for A \neq \emptyset,\\
\eeq
\beq\label{eqN2}
N^{\emptyset} = (\sum_{C}M^{C}(M^{\emptyset})^{-1}M^{C})^{-1}.\\
\eeq
\end{lemma}
\begin{proof}
In fact, assume the octonionic matrix $N^{A, B}= (A.B|B)N^{A.B} = (A.B|A)\widetilde{N}^{A.B}$ is the inverse of $M$, where$$N^{A.B}=(A.B|A.B)\widetilde{N}^{A.B} = \left\{
    \begin{array}{ll}
      -\widetilde{N}^{A.B}, & \hbox{$A.B \neq \emptyset$;} \\
      \widetilde{N}^{\emptyset}, & \hbox{$A.B = \emptyset$.}
    \end{array}
  \right.
$$ Then
\beq\label{eq3}
\sum_{C}(A.C|A)(C.B|B)\widetilde{N}^{A.C}M^{C.B} = 0, \ \ for A \neq B,\\
\eeq
\beq\label{eq4}
\sum_{C}(A.C|A)(C.A|A)\widetilde{N}^{A.C}M^{C.A} = \Id. \\
\eeq
Changing $C$ into $A.B.C$ in $(\ref{eq3})$, we get $\sum_{C}(B.C|A)(A.C|B)\widetilde{N}^{B.C}M^{C.A} = 0$, then it is enough to have  $$(A.C|A)(C.B|B)\widetilde{N}^{A.C}M^{C.B} + (B.C|A)(A.C|B)\widetilde{N}^{B.C}M^{C.A} = 0.$$

According to Lemma~\ref{altern}, it holds as soon as
\beq\label{eq1}
\widetilde{N}^{A.C}M^{C.B} = \widetilde{N}^{B.C}M^{C.A}.
\eeq
Choosing $C=B$ and then setting $D = A.C$, this leads to
\beq\label{eqN}
\widetilde{N}^{D} = N^{\emptyset}M^{D}(M^{\emptyset})^{-1},
\eeq
for every $D \in \cP(E)$. Now choose $C = \emptyset$ in $(\ref{eq1})$ and apply $(\ref{eqN})$ to $N^{A}$ and $N^{B}$, plugging into $(\ref{eq1})$, this reduces to
$$
N^{\emptyset}M^{A}(M^{\emptyset})^{-1}M^{B} = N^{\emptyset}M^{B}(M^{\emptyset})^{-1}M^{A}.
$$
Now $(\ref{eq4})$ is
$$
\sum_{D}\widetilde{N}^{D}M^{D} = \Id, \
$$
and using $(\ref{eqN})$, this  gives
$$
N^{\emptyset}(\sum M^{C}(M^{\emptyset})^{-1}M^{C})^{-1} = \Id,
$$
which means that $N^{\emptyset}$ is invertible, and gives its inverse, such that $(\ref{eqN2})$ holds true. Then we can use $(\ref{eqN2})$ and $(\ref{eqN})$ to get $(\ref{eqN1})$.

\end{proof}

\begin{remark}  It is worth to observe for later use that if the matrix $\cM$ on the octonions satisfies the assumptions of Lemma~\ref{lemma.invert}, then it is also the case of $\cM-X\Id$.

\end{remark}
\section{Symmetric diffusion operators on matrices}

We introduce the basics on symmetric diffusion operators, in a simplified version adapted to our case. For further details see \cite{bglbook}.

Let $E$ be an open set in $\bbR^n$, endowed with a $\sigma$-finite measure $\mu$  and let $\cA_0$ be the set of smooth compactly supported functions,  or of polynomials functions on $E$.
For any linear operator $\LL : \cA_0\mapsto \cA_0$, we define its carré du champ operator as
$$\Gamma(f,g) = \frac{1}{2}\Big( \LL(fg)-f\LL(g)-g\LL(f)\Big)\,.$$
We have the following
\bdefi \label{symmopt}
A symmetric diffusion operator is a  linear operator $\LL$: $\cA_0\oplus 1\mapsto \cA_0$, such that
\begin{enumerate}
\item $ \LL(1)=0$,
\item $\forall f,g\in \cA_0\oplus 1, ~\int f\LL(g)\, d\mu = \int g \LL (f) \, d\mu$,
\item $\forall f\in \cA_0, \Gamma(f,f) \geq 0$,
\item $\forall f=(f_1, \cdots, f_n)$, where $f_i\in \cA_0$ , $\Phi$ is a smooth function $\bbR^n\mapsto \bbR$ and $\Phi(0)=0$,
\beq\label{formulaL}
\LL(\Phi(f))= \sum_i \partial_i \Phi(f) \LL(f_i) + \sum_{i,j} \partial^2_{ij} \Phi(f) \Gamma(f_i,f_j).
\eeq
\end{enumerate}
\edefi

Consider an open set $\Omega\subset E$, and a given system of coordinates $(x^i)$, then we can write
$$\LL(f) = \sum_{ij} g^{ij} (x)\partial^2_{ij} f + \sum_i b^i(x)\partial_i f,$$ where
$$g^{ij}(x)= \Gamma(x^i, x^j), ~b^i (x)= \LL(x^i).$$

In this paper, we perform computations on the characteristic polynomial $P(X) = \det(M - XId)$ of a matrix $M$. Assume that we have some diffusion operator acting on the entries of a matrix $M$, described by the values of $\LL(M_{ij})$ and $\Gamma(M_{ij}, M_{kl})$ for any $(i,j,k,l)$. Then, we have,
\beqnas
\Gamma(\log P(X), \log P(Y)) &=& \sum_{i,j,k,l} \partial_{M_{ij}}\log(P(X))\partial_{M_{kl}}\log(P(Y))\Gamma(M_{ij}, M_{kl}),\\
\LL(\log P(X)) &=& \sum_{i,j} \partial_{M_{ij}}\log(P(X)) \LL(M_{ij}) + \sum_{i,j,k,l} \partial_{M_{ij}}\partial_{M_{kl}}\log(P(X))\Gamma(M_{ij}, M_{kl}).\\
\eeqnas
To compute $\partial_{M_{ij}}\log(P(X))$ and $\partial_{M_{ij}}\partial_{M_{kl}}\log(P(X))$ in the above formulae, we use the Lemma~6.1 in Bakry and Zani \cite{BakZ}, which we quote here without proof.

\begin{lemma}\label{det.derivative}
Let $M = (M_{ij})$ be a matrix and $M^{-1}$ be its inverse, on the set $\{\det M \neq 0\}$ we have
\begin{eqnarray*}
\partial_{M_{ij}}\log \det M &=& M^{-1}_{ji},\\
\partial_{M_{ij}}\partial_{M_{kl}}\log \det M &=& -M^{-1}_{jk}M^{-1}_{li}.\\
\end{eqnarray*}
\end{lemma}
Hence,  with $M^{-1}(X) = (M - XId)^{-1}$,
\beqna\label{eq.character}
&&\label{eq.character1}\Gamma(\log P(X), \log P(Y)) = \sum_{i,j,k,l} M^{-1}(X)_{ji}M^{-1}(Y)_{lk}\Gamma(M_{ij}, M_{kl}),\\
&&\label{eq.character2}\LL(\log P(X)) = \sum_{i,j} M^{-1}_{ji}(X)\LL(M_{ij}) - \sum_{i,j,k,l} M^{-1}_{jk}(X)M^{-1}_{li}(X)\Gamma(M_{ij}, M_{kl}).
\eeqna

According to Bakry and Zani \cite{BakZ}, one can get  from $\Gamma(P(X), P(Y))$ and $\LL(P(X))$ informations  about the multiplicities of the eigenvalues,  and on  the invariant measure of the operator $\LL$ acting on $P(X)$:\\
If for some constants $\alpha_1, \alpha_2, \alpha_3$,
\beq\label{eq.gal.P}\LL(P)= \alpha_1 P''+ \alpha_2 \frac{P'^2}{P}, ~\Gamma(\log P(X), \log P(Y))=  \frac{\alpha_3}{Y-X} \Big( \frac{P'(X)}{P(X)}-\frac{P'(Y)}{P(Y)}\Big).\eeq
and if there exists   for some $a \in \bbR$, $a\neq 0$ which satisfies
\beq \label{eq.puiss}a^2(\alpha_1+\alpha_2)-a(\alpha_1+ \alpha_3) + \alpha_3=0,\eeq
Then
\begin{enumerate}
\item If $a$ is a positive integer, it is the multiplicity of the eigenvalues of $M$;
\item Write $P(X) = \prod^{n}_{i=1}(X-x_i)^a$, the invariant measure for the operator $\LL$ in the  Weyl chamber $\{ x_1 < ... <x_n\}$  is, up to a multiplicative constant, $$d\mu = (\prod_{i < j}(x_{i} - x_{j})^2)^{-\frac{a^2(\alpha_1 + \alpha_2)}{\alpha_3}}d\mu_0,$$
    where $d\mu_0$ is the Lebesgue measure.
\end{enumerate}

\section{Symmetric matrices on octonions}
Our aim is to describe the law of the spectrum of the real form of symmetric matrices on octonions. The block matrix is $\mathcal{M} = ((A.B|B)M^{A.B})_{A, B \in \cP(E)}$, satisfying $(M^{A})^{t} = (A|A)M^{A}$ from the  symmetry  assumption.

We will focus on  cases where the symmetry condition $(\ref{symm})$ of matrix $\mathcal M - X\mathrm{Id}$ is satisfied, i.e. where the matrix
$$U(X) := (\mathcal M - X \mathrm{Id})^{-1}$$ is octonionic (almost surely for the stochastic process under consideration).

Setting $$P(X) := \det (\mathcal M - X\mathrm{Id}),$$
by Lemma~\ref{det.derivative} we have
\beqna
\nonumber\Gamma(\log P(X), \log P(Y)) &=& \sum_{\substack{A,B,C, D \\ i,j, k,l}} U_{ji}^{B,A} U_{lk}^{D,C}\Gamma(M_{ij}^{A,B}, M_{kl}^{C,D})\\
&=& \label{gamma-1} \sum_{\substack{A,B,C, D \\ i,j, k,l}} (A.B|A.B) (C.D|C.D) U_{ji}^{A.B} U_{lk}^{C.D}\Gamma(M_{ij}^{A.B}, M_{kl}^{C.D})
\eeqna
(where we used property 2 of Lemma 2.3), and
\beq\label{L0}
\LL(\log P(X)) = \sum_{\substack{A,B\\ i,j}} U^{A, B}_{ji}(X)\LL(M^{A, B}_{ij}) - \sum_{\substack{A,B,C, D \\ i,j, k,l}} U^{B, C}_{jk}(X)U^{D, A}_{li}(X)\Gamma(M^{A, B}_{ij}, M^{C, D}_{kl}).\\
\eeq

For further use in both examples, we state (without proof) two preliminary lemmas. The first lemma collects some elementary facts, consequences of the definition, the symmetries and property 2 of Lemma~\ref{altern}.
\begin{lemma}\label{newlem}
For $F \in \cP(E)$, we have
\beqna
\label{L1}
\tr\!\ U(X)^F &=& \sum_i U(X)^F_{ii}\,,\\
\label{L2}
\tr \!\ U(X) &=& 8  \!\ \tr\!\ U(X)^\emptyset\,,\\
\label{L3}
\tr\!\ [U(X)^F U(Y)^F]\ &=&\sum_{ij}   U(X)^F_{ji}U(Y)^F_{ij}\,, \\
\label{L4}
 (F|F) \tr\!\ [U(X)^F U(Y)^F] &=& \sum_{ij} U(X)^F_{ji}U(Y)^F_{ji}\,, \\
\label{L5}
\tr \!\ [U(X) U(Y)]  &=& 8\sum_C (C|C) \tr\!\  [U(X)^C U(Y)^C]\,.
\eeqna
\end{lemma}
The second lemma gives expressions of the traces in terms of the characteristic polynomials. The first two identities  are obtained by derivation from
\[\log P(X) = \tr\!\ \log (\mathcal M - X\hbox{Id}) = - \tr\!\ \log U(X)\,,\]
and the third one  is a consequence of the first one and the resolvent equation.

\begin{lemma}\label{newnewlem}
\beqna
\label{P1}
\tr\!\ U(X) &=& \frac{P'(X)}{P(X)}\,,\\
\label{P2}
\tr\!\ (U(X)^2) &=&\frac{P'(X)^2}{P(X)^2}  - \frac{P''(X)}{P(X)}\,,\\
\label{P3}
\tr\!\ [U(X)U(Y)] &=& \frac{1}{Y-X}\left(\frac{P'(X)}{P(X)} - \frac{P'(Y)}{P(Y)}\right)\,.
\eeqna
\end{lemma}

\subsection{The dimension 2 case}
Consider $M = \sum M^{A}\omega_{A}$, where $\{M^{A}\}$ are matrices whose elements are independent Brownian motions. For $A \neq \emptyset$, due to symmetry of $\mathcal{M}$, $(M^{A})^{t} = (A|A)M^{A}= -M^A$. Such matrices naturally satisfy the symmetry restriction~\ref{symm} in dimension 2, since the $2 \times 2$ antisymmetric matrices are all of the form $\left(
                                 \begin{array}{cc}
                                   0 & -z \\
                                   z & 0 \\
                                 \end{array}
                               \right)$, and they are therefore all proportional to each other. This  property fails  in higher dimensions. Set
\beq\label{basis}
\Gamma(M^{A}_{ij}, M^{B}_{kl}) = \frac{1}{2}\delta_{A, B}(\delta_{ik}\delta_{jl} + (A|A)\delta_{il}\delta_{jk}), \ \ \ L(M^{A}_{ij}) = 0,
\eeq
which reflects the symmetry of the matrices. Notice that the inverse matrix $U(X)$ is also symmetric with $(U^{A})^{t} = (A|A)U^{A}$. We have the following result

\begin{proposition}For the $2 \times 2$ symmetric matrix $M = \sum M^{A}\omega_{A}$,
\beq\label{gamma}
\Gamma(\log P(X), \log P(Y)) = \frac{8}{Y - X}\left(\frac{P'(X)}{P(X)} - \frac{P'(Y)}{P(Y)}\right),
\eeq
\beq\label{L}
L(\log P(X)) = 3\left(\frac{P'(X)^2}{P(X)^2} - \frac{P''(X)}{P(X)}\right) - \frac{1}{2}\frac{P'(X)^2}{P(X)^2}.\\
\eeq
\end{proposition}
\begin{proof}
Proof of (\ref{gamma}) : From (\ref{gamma-1}) and (\ref{basis}) we have
\beqnas
& &\Gamma(\log P(X), \log P(Y))\\
&=& \sum_{\substack{A,B,C, D \\  AB= CD}}(A.B|A.B)(C.D|C.D)U(X)^{B.A}_{ji}U(Y)^{D.C}_{lk}\frac{1}{2}(\delta_{ik}\delta_{jl} + (A.B|A.B)\delta_{il}\delta_{jk})\\
&=& 8\sum_{F}{\rm tr}(U(X)^{F}U(Y)^{F})(F|F) \\
&=& 8{\rm tr}(U(X)U(Y)) = \frac{8}{Y - X}\left(\frac{P'(X)}{P(X)} - \frac{P'(Y)}{P(Y)}\right),\\
\eeqnas
where we applied (\ref{L4}), (\ref{L3}), (\ref{L5}) and formula 3 in Lemma~\ref{newnewlem}. This ends the proof of (\ref{gamma}).

Proof of (\ref{L}) :  In the following, we write $U$ is  for $U(X)$.  From (\ref{L0}) and (\ref{basis}), we have
\beqnas
\LL(\log  P(X)) &=& -\frac{1}{2}\sum_{B,C,D} (B.D|B.D) \sum_{ij} (U_{ji}^{B.C})^2  \\
\label{logP}
&& -\frac{1}{2} \sum_{\substack{A,B,C, D \\  A.B.C.D = \emptyset}} (B.D|B.D)(A.B|A.B) (\sum_i U_{ii}^{BC})^2.
\eeqnas
On the one hand, in view of (\ref{L4}) and (\ref{L5})
\beqnas
\sum_{B,C,D} (B.D|B.D) \sum_{ij} \left(U_{ji}^{B.C}\right)^2  = -6\!\  \tr\!\ (U^2)\,.
\eeqnas
On the other hand, since for $F \not= \emptyset$, $U^{F}$ is antisymmetric, so that $ \hbox{tr}\!\  U^{F} = 0$,
\beqnas
\displaystyle \sum_{\substack{A,B,C, D \\  A.B.C.D = \emptyset}} (B.D|B.D)(A.B|A.B) (\sum_i U_{ii}^{B.C})^2 = 8^2 \left(\tr\!\ U^\emptyset\right)^2.
\eeqnas
Combined as in (\ref{logP}), these sums give :
\begin{equation}
 L(\log P(X)) = 3 \hbox{tr}\!\ U^2 - 4\cdot 8  \left( \hbox{tr}\!\  U^{\emptyset}\right)^2 =  3 \hbox{tr}\!\ U^2 - \frac{1}{2}  \left( \hbox{tr}\!\  U\right)^2,
\end{equation}
which, in view of (\ref{P1}) and (\ref{P2}) ends the proof of (\ref{L}) and then the proof of the Proposition.
\end{proof}

\begin{remark}s
By the results of above proposition and formula $(\ref{formulaL})$, it is easy to get
\begin{eqnarray*}
L(P) =  (11-\frac{1}{2})\frac{P'(X)^2}{P(X)} - 11P''(X).
\end{eqnarray*}
Now chose  $\alpha_1 = -11$, $\alpha_2 = 11 - \frac{1}{2}$ and  $\alpha_3 = 8$ in  formula $(\ref{eq.gal.P})$ : the resulting value for $a$ is  $a = 8$.  This shows that the multiplicity of the eigenvalues is 8.\\
Assume $\rho$ is the density of the invariant measure of $L$ of the coordinates $\{x_{i}\}$ in the Weyl chamber, then according to our discussion in the previous section, we have
$$
\rho = C\prod_{i < j}(x_{i} - x_{j})^{8}.
$$
\end{remark}

\subsection{Another model in any dimension}

We now provide another set of random octonionic  matrices  for which the symmetry condition $(\ref{symm})$ is automatically satisfied.

Let $M^{\emptyset}$ be a symmetric matrix with independent Brownian motions as its entries. For all $A, B \neq \emptyset$, let $M^{A} = M^{B} = \cA$ be a random antisymmetric matrix with independent Brownian motion as its off diagonal entries. Then consider $M = M^{\emptyset}\omega_{\emptyset} + \cA\sum_{C \neq \emptyset}\omega_{C}$. This model is similar to the Hermitian case considered in Bakry and Zani~\cite{BakZ} (see Remark~\ref{Hermitian}).
Similarly to the Hermitian case, we set
\beq\label{gamma-2.1}
\Gamma(M^{\emptyset}_{ij}, M^{\emptyset}_{kl}) = \frac{1}{2}(\delta_{ik}\delta_{jl} + \delta_{il}\delta_{jk}),
\eeq
\beq\label{gamma-2.2}
\Gamma(\cA_{ij}, \cA_{kl}) = \frac{1}{14}(\delta_{ik}\delta_{jl} - \delta_{il}\delta_{jk}),
\eeq
\beq\label{L2}
\Gamma(M^{\emptyset}_{ij}, \cA_{kl}) = 0, \ \ \ \LL(M^{A}_{ij}) = 0.
\eeq
Due to Lemma~\ref{altern}, for the inverse matrix $U(X) = (M - XI)^{-1}$, we have for every $C \neq \emptyset$
$$
U^{C} = -U^{\emptyset}M^{C}(M^{\emptyset}-XI)^{-1} = -U^{\emptyset}\cA(M^{\emptyset}-XI)^{-1},
$$
which means for all $C \neq \emptyset$, $U^{C}$ is the same, and we denote it by $U_{a}$.

\begin{proposition}For the matrix $M = M^{\emptyset}\omega_{\emptyset} + \sum_{C \neq \emptyset}\cA\omega_{C}$ on the octonions,
\begin{eqnarray*}
\Gamma(\log P(X), \log P(Y))&=& \frac{8}{Y - X}\left(\frac{P'(X)}{P(X)} - \frac{P'(Y)}{P(Y)}\right),\\
L(\log P) &=& - \frac{1}{8}\frac{P'(X)^2}{P(X)^2}.\\
\end{eqnarray*}
\end{proposition}
\begin{proof}
On the one hand, from (\ref{gamma-1}) and (\ref{L2})
\beq\label{gamma11}
\Gamma (\log P(X), \log P(Y)) = \frac{1}{2} S_1  + \frac{1}{14} S_2,
\eeq
where
\beqnas
S_1 &:=& \sum_{A=B, C=D}\sum_{i,j,k,l} U(X)^\emptyset_{ji} U(Y)^\emptyset_{lk}(\delta_{ik}\delta_{jl} + \delta_{il}\delta_{jk}),\\
S_2 &:=& \sum_{A \not=B, C\not= D} \sum_{i,j,k,l} U(X)^{A.B}_{ji} U(Y)^{C.D}_{lk}(\delta_{ik}\delta_{jl} - \delta_{il}\delta_{jk})\,.
\eeqnas
A careful computation, using the fact that  $U(Y)^\emptyset$ is symmetric and $U(Y)_a$ is antisymmetric gives
\beq\label{sprime1}
S_1 = 2\cdot 8^2\!\  \tr\!\ \left[U(X)^\emptyset U(Y)^\emptyset\right]\ , \ S_2 = - 2 \cdot 7^2 \cdot 8^2  \!\  \tr\!\ \left[U(X)_a U(Y)_a\right]\,,
\eeq
and then, using (\ref{gamma11})
\[\Gamma (\log P(X), \log P(Y)) = 8^2  \!\  \tr\!\ \left[U(X)^\emptyset U(Y)^\emptyset \right] - 7\cdot 8^2  \!\ \tr\!\ \left[U(X)_a U(Y)_a \right]. \]
Going back to (\ref{L5}) we see that
\begin{eqnarray}
\hbox{tr}\!\ \left[U(X) U(Y)\right]
&=& 8 \!\  \tr\!\ \left[U(X)^\emptyset U(Y)^\emptyset \right] - 7 \cdot 8 \!\ \tr\!\ \left[U(X)_a U(Y)_a \right],
\end{eqnarray}
so that
\[\Gamma(\log P(X) , \log P(Y)) = 8 \!\ \hbox{tr}\!\ \left[U(X) U(Y)\right] = \frac{8}{Y-X}\left(\frac{P'(X)}{P(X)}- \frac{P'(Y)}{P(Y)}\right)\,.\]

On the other hand, with $U$ for $U(X)$
\beqnas
\LL(\log P) &=& -\frac{1}{2} \displaystyle \sum_{\substack{A=B,C=D \\  i,j,k,l}}(B.C|C)(D.A|A)(A.B|B)(C.D|D)U(X)^{B.C}_{jk}U(X)^{D.A}_{li}(\delta_{ik}\delta_{jl} + \delta_{il}\delta_{jk}) \\
&& - \frac{1}{14}  \displaystyle \sum_{\substack{A\not =B,C\not=D \\  i,j,k,l}}(B.C|C)(D.A|A)(A.B|B)(C.D|D)U(X)^{B.C}_{jk}U(X)^{D.A}_{li}(\delta_{ik}\delta_{jl} - \delta_{il}\delta_{jk})\\
\label{bigL}
&=:& -\frac{1}{2} S'_1 - \frac{1}{14} S'_2\,.
\eeqnas
Let us first remark that
\beq\label{sumijkl}
\sum_{ijkl} U_{jk}^F U_{li}^G (\delta_{ik}\delta_{jl} \pm \delta_{il}\delta_{jk}) = (F|F) \tr\!\ (U^F U^G) \pm (\tr\!\ U^F) (\tr\!\ U^G).
\eeq
For the first part, we have
\beqnas
S'_1 &=& \sum_{B,C}\tr\!\ (U^{B.C})^2 + \sum_{B,C} (B.C|B.C) [\tr\!\ U^{B.C}]^2\\ \label{sprime}
&=& 8\left[\tr\!\ (U^\emptyset)^2 + 7 \tr\!\ (U_a)^2\right]  + 8  [\tr\!\ U^\emptyset]^2.
\eeqnas

For the second part, going back to the notation (\ref{theta}) for $\Theta$ and applying (\ref{sumijkl}) we have
\begin{eqnarray}
\label{sprime2}
S'_2 = \sum_{A \not= B, C \not= D} \Theta\times \left((B.C|B.C)\tr\!\ [U^{B.C}  U^{A.D}] - \tr\!\ U^{B.C} \tr\!\ U^{A.D}\right).
\end{eqnarray}
Let us split the sum into four parts according to $B=C$ or not,  and $A=D$ or not.

i) When $A \not = D , B=C$, the sum vanishes. Indeed, in this case,
\beq\label{UU}
(B.C|B.C)\tr\!\ [U^{B.C}U^{A.D}]- (\tr\!\ U^{B.C}\tr\!\ U^{A.D}) = \tr [U^\emptyset U_a],
\eeq
and $\Theta = (A.D|A)(A.B|B)(B.D|D)$ which is antisymmetric in  $A,B$.

ii) The same occurs when $A=D$ and $B \not = C$.

iii) When $A=D$ and $B=C$,
\[(B.C|B.C)\tr\!\ [U^{B.C}U^{A.D}]- (\tr\!\ U^{B.C}\tr\!\ U^{A.D})  = \tr\!\ [U^\emptyset]^2 - (\tr\!\ U^\emptyset)^2\,,\]
and $\Theta = (B.D|B.D)$ in view of property 4 of Lemma 2.3,
so that  the contribution is
\begin{eqnarray*}
& &\sum_{A \not= B, C \not= D, B=C, A=D}(B.D|B.D)\left((B.C|B.C)\tr\!\ [U^{B.C}  U^{A.D}] - \tr\!\ U^{B.C} \tr\!\ U^{A.D}\right)\\
&=& -  7\cdot 8\left(\tr\!\ [U^\emptyset]^2 - (\tr\!\ U^\emptyset)^2\right)\,.
\end{eqnarray*}

iv) When $A\not= D, B\not=C$,
\[(B.C|B.C)\tr\!\ [U^{B.C}U^{A.D}]- (\tr\!\ U^{B.C}\tr\!\ U^{A.D})  = - \tr\!\ (U_a)^2\,.\]

With the help of some computer algebra, we get
\[ \sum_{A \not= B, C \not= D, B \not=C, A\not= D} \Theta = 2^3\cdot 7^2\,, \]
and finally, all the contributions in (\ref{sprime2}) give
so that
\begin{eqnarray*}
S'_2 =  - 7\cdot 8 \left(\tr\!\ [U^\emptyset]^2 - (\tr\!\ U^\emptyset)^2\right) - 7^2\cdot 8\!\ \tr\!\ (U_a)^2\,.
\end{eqnarray*}
Going back to (\ref{bigL}) and (\ref{sprime}), we conclude, using again (\ref{L2})
\[\LL(\log P) = - 8 (\tr\!\ U^\emptyset)^2\ = - \frac{1}{8} (\tr\!\ U)^2\,,\]
which ends the proof of the proposition.
\end{proof}

\begin{remark}
Similarly we have
$$
\LL(P(X)) = (8 - \frac{1}{8})\frac{P'(X)^2}{P} - 8P''(X).\\
$$
Chose  $\alpha_1 = -8$, $\alpha_2 = 8 - \frac{1}{8}$, $\alpha_3 = 8$, we still obtain  the multiplicity $a = 8$, while the density of the invariant measure of $L$ is then  $$C \prod_{i < j}|x_{i} - x_{j}|^{2}.$$
\end{remark}

\begin{remark}\label{Hermitian}
Recall that in Bakry and Zani~\cite{BakZ}, section 7.1, for a Hermitian matrix $H = M + iA$ with independent Brownian motions as its entries (where $M$ is symmetric, $A$ is anti-symmetric), we have
\beqnas
\Gamma(M_{ij}, M_{kl}) &=& \frac{1}{2}(\delta_{ik}\delta_{jl} + \delta_{il}\delta_{jk}),\\
\Gamma(A_{ij}, A_{kl}) &=& \frac{1}{2}(\delta_{ik}\delta_{jl} - \delta_{il}\delta_{jk}).\\
\eeqnas

In our model $M = M^{\emptyset}\omega_{\emptyset} + \cA\sum_{C \neq \emptyset}\omega_{C}$, denote $e$ the specific element in the octonion algebra,  $e = \sum_{C \neq \emptyset}\omega_{C}$. Notice that
$$
e^2 = -7,
$$
which indicates that $e$ works like $i$ in the Hermitian matrices, just with a different variance. Therefore,  this example is indeed similar
to  the case of Hermitian matrices.
\end{remark}

\begin{remark}
In this remark we would like to discuss why the dimension 2 is so special. Consider a more general model: let
$$
M = M^{\emptyset}\omega_{0} + \sum_{C \neq \emptyset}M^{C}\omega_{C},
$$
with $M^{C} = x_{C}A_{0}$, where $M^{\emptyset}$ is a Brownian motion on symmetric matrices, $\{x_{C}\}$ a series of Brownian motions on $\mathbb{R}$, and $A_{0}=\{a_{ij}\}$ a fixed anti-symmetric matrix. Obviously this model satisfies the symmetry condition~\ref{symm}. When $M$ is a $2 \times 2$ matrix, it can be considered as a special example of the first case. Let $e = \sum_{C \neq \emptyset}x_{C}\omega_{C}$. Different from the previous model, in this case $e$ can be considered as a Brownian motion on the basis of octonions satisfying $e^2 = -\sum_{C \neq \emptyset}|x_{C}|^2 = -|e|^2$.

Therefore,
\begin{eqnarray*}
\Gamma(M^{\emptyset}_{ij}, M^{\emptyset}_{kl}) &=& \frac{1}{2}(\delta_{ik}\delta_{jl} + \delta_{il}\delta_{jk}),\\
\Gamma(M^{A}_{ij}, M^{B}_{kl}) &=& \delta_{A=B}a_{ij}a_{kl},  A, B \neq \emptyset,\\
\Gamma(M^{\emptyset}_{ij}, M^{A}_{kl}) &=& 0,  A \neq \emptyset.\\
\end{eqnarray*}
Similar computations yield
\begin{eqnarray*}
\Gamma(\log P(X), \log P(Y)) &=& 8\sum_{A.B = \emptyset}{\rm tr}(U(X)^{\emptyset}U(Y)^{\emptyset})+ 8\sum_{A.B \neq \emptyset}({\rm tr}(U(X)^{B.A}A_{0}){\rm tr}(U(Y)^{B.A}A_{0})),\\
L(\log P) &=& - \frac{1}{2}\sum_{A.C \neq \emptyset}{\rm tr}(U(X)^{A.C})^2 -4{\rm tr}(U(X)^{\emptyset})^2 - 4({\rm tr}U(X)^{\emptyset})^2 \\
& & + 5\sum_{B.C \neq \emptyset}{\rm tr}(U(X)^{B.C}A_{0}U(X)^{B.C}A_{0}) + 56{\rm tr}(U(X)^{\emptyset}A_{0}U(X)^{\emptyset}A_{0}),\\
\end{eqnarray*}
which are hard to describe in terms of $P$.
When the matrix is $2 \times 2$, the following equalities hold: for $C \neq \emptyset$,
\begin{eqnarray*}
{\rm tr}(U(X)^{C}A_{0}){\rm tr}(U(Y)^{C}A_{0}) &=& - {\rm tr}(U(X)^{C}U(Y)^{C}),\\
{\rm tr}(U(X)^{C}A_{0}U(X)^{C}A_{0}) &=& -{\rm tr}(U(X)^{C})^2,\\
{\rm tr}(U(X)^{\emptyset}A_{0}U(X)^{\emptyset}A_{0}) &=& {\rm tr}(U^{\emptyset})^2 - ({\rm tr}U^{\emptyset})^2.\\
\end{eqnarray*}
which give rise to the results in the first model.

However, in higher dimensions, the above conditions are hard to satisfy.
In fact when $n = 2$, it is enough to take $A_{0} = \left(
           \begin{array}{cc}
             0 & -1\\
             1 & 0 \\
           \end{array}
         \right)$. Set $x_{\{1\}} = z$. By the formula
\begin{eqnarray*}
U^{\emptyset} &=& (M^{\emptyset} + z^{2}A_{0}(M^{\emptyset}){-1}A_{0})^{-1},\\
U^{\{1\}} &=& -zU^{\emptyset}A_{0}(M^{\emptyset}){-1}.\\
\end{eqnarray*}
It is easily seen that $U^{\{1\}}$ is a $2 \times 2$ antisymmetric matrix and can be written as
$$
U^{\{1\}} = \lambda A_{0},
$$
where $\lambda = \frac{z}{z^2-\det(M^{\emptyset})}$. Compare it with the expressions of $U^{\emptyset}$ and $U^{\{1\}}$, we have $A_{0} = \frac{\lambda}{\lambda z^2 - z}M^{\emptyset}A_{0}M^{\emptyset}$. Since $A_{0}^2 = -I$, this leads to $(M^{\emptyset}A_{0})^2 = \frac{z-\lambda z^2}{\lambda}I = -\det(M^{\emptyset}) I$, which is impossible to hold for any symmetric matrix $M^{\emptyset}$ in higher dimensions.  This restriction insures  the first two conditions, and the third one is proved by this and the fact that ${\rm tr}(M^2) - ({\rm tr}M)^2 = -2\det(M)$ holds in dimension 2.
\end{remark}

\section{Some remarks}
Our two models provide examples where the multiplicity of eigenvalues and the exponent $\beta$ in the law are not related, which is in accordance with the conclusion in Bakry and Zani~\cite{BakZ}, that the exponent reflects the structure of the algebra while the multiplicity of the eigenvalues is decided by the dimension of the eigenspaces.

As we have seen, the octonionic structure of the matrix plays an important role. For higher dimension, the problem may be studied by our method if we know the structure of the inverse matrix, which is not necessarily octonionic. The main obstacle is still the non-associativity, which prevents any  matrix presentation for octonionic multiplication. Let us recall that the $3 \times 3$ matrices on octonions have been studied by Dray and Manogue~\cite{drayM} and Okubo~\cite{oku} using algebraic method, showing that there are 6 eigenvalues with multiplicity 4.  It it still an open problem to provide a probabilistic model in this case which would lead to  this conclusion.

\section{Acknowledgement}

I would like to thank my supervisor, Prof. Dominique Bakry, for his precious suggestions and helpful discussions. I am also indebted to my other supervisors, Prof. Jiangang Ying and Prof. Xiangdong Li, for their support and help. I would also like to express my gratitude to the referee for careful reading and suggestions. 

This research was supported by China Scholarship Council.

\bibliographystyle{amsplain}
\bibliography{ref_oct}

\end{document}